\documentclass[11pt]{amsart}
\usepackage[utf8]{inputenc}
\usepackage{amsmath, amsfonts, amssymb, amsthm, amsopn}
\usepackage{stmaryrd}
\usepackage{url}
\usepackage[cmtip, all]{xy}

\usepackage{xcolor}

\usepackage{graphicx}
\usepackage{verbatim}

\newcommand{\ga}[2]{\begin{gather}\label{#1}#2 \end{gather}}

\theoremstyle{plain}
\newtheorem{theorem}{Theorem}[section]

\newtheorem{corollary}[theorem]{Corollary}

\theoremstyle{definition}

\newtheorem{example}[theorem]{Example}
\newtheorem{remark}[theorem]{Remark}

\newcommand{\CC}{\mathbb{C}}

\newcommand{\NN}{\mathbb{N}}

\newcommand{\QQ}{\mathbb{Q}}

\newcommand{\ZZ}{\mathbb{Z}}
\newcommand{\FF}{\mathbb{F}}

\newcommand{\E}{\mathcal{E}}

\newcommand{\K}{\mathcal{K}}
\newcommand{\V}{\mathcal{V}}
\newcommand{\sL}{\mathcal{L}}

\newcommand{\sV}{\mathcal{V}}

\newcommand{\sO}{\mathcal{O}}

\renewcommand{\setminus}{\smallsetminus}

\let\lim=\relax
\DeclareMathOperator*{\lim}{lim}

\begin{document}

\title[Moduli spaces of local systems]{Survey   on special subloci of the moduli spaces of  local systems on complex varieties }
\author{H\'el\`ene Esnault }
\address{ The Institute for Advanced Study, School of Mathematics, 1 Einstrin Dr., Princeton, NJ 08540, USA}
\email{esnault@ias.edu}
\address{Freie Universit\"at Berlin, Arnimallee 3, 14195, Berlin,  Germany}

\email{esnault@math.fu-berlin.de}

\thanks{The  author gratefully acknowledges  the support of  the  Institute for Advanced Study in Princeton where this report was written}

\begin{abstract}  It is a short report on recent results concerning special loci of the Betti moduli space of irreducible complex local systems on complex varieties. 

\end{abstract}

\maketitle

\section{Introduction}\label{sec:intro}

On   a smooth complex connected quasi-projective variety $X$, with underlying complex analytic manifold $X^{\rm an}$, 
one has the category
 ${\sf LocSys}(X^{\rm an})$ of complex local systems $\sV$ on $X^{\rm an}$. A complex point $a\in X(\CC)$ yields the structure of a Tannakian category where the neutral fiber functor is given by restriction $\sV|_a$. The Riemann-Hilbert correspondence establishes an equivalence of Tannakian categories between ${\sf LocSys}(X^{\rm an})$ and ${\sf Conn}^{\rm reg}(X)$, the categories of integrable connections $(E,\nabla)$ on $X$ which are regular singular at infinity, where the neutral fibre functor is given by restriction $E|_a$. The correspondence is defined by $(E,\nabla)\mapsto E^\nabla, \ \sV\mapsto (\sV\otimes_{\CC} \sO_{X^{\rm an}}, 1\otimes d) $ on $X^{\rm an}$. 
 One has then to descend the  analytic integrable connection to an algebraic one on $X$. This is the content of Deligne's Riemann-Hilbert correspondence 
  \cite[Thm.5.9]{Del70}.   
 
 On the other hand, in a given rank $r\in \NN_{>0}$, in case $X$ is projective,  Simpson \cite{Sim94} defines coarse moduli spaces  $\mathbf{{M}_B} (X,r)$  for  rank $r$ complex local systems and $\mathbf{{M}_{dR}} (X,r)$  for rank $r$ integrable connections. 
The Riemann-Hilbert correspondence establishes a complex analytic isomorphism between $\mathbf{{M}_B} (X,r)(\CC)$ and $\mathbf{{M}_{dR}} (X,r)(\CC)$.

Simpson considers $0$-dimensional  subloci in  $  \mathbf{{M}_B} (X,r)^{\rm irr}\subset  \mathbf{{M}_B} (X,r)$, the open of irreducible local systems. The underlying complex points are called {\it rigid local systems}. He conjectures  \cite[Conjecture p.9]{Sim92}  that they are {\it motivic}, that is  direct factors of Gau{\ss}-Manin local systems of  smooth projective families over some dense open of $X$, as is the case when $X$ has dimension one (\cite{Kat96}). He remarks \cite{Sim92}, {\it loc. cit.}  that his motivicity conjecture  in particular implies  the {\it integrality conjecture} predicting that  rigid local systems are integral, that is all the traces of the underlying representation of $\pi_1(X^{\rm an})$ of a rigid local system lie in $\bar \ZZ$. 
He extends his motivicity and  integrality conjectures   to the case when $X$ is quasi-projective.

In Section~\ref{sec:int} we report on our solution with Michael Groechenig   of the integrality conjecture for the $0$-dimensional components which have a reduced scheme structure. The corresponding rigid local systems are called {\it cohomologically rigid} as the condition is equivalent to saying that the Zariski tangent space, which is a cohomology group, is equal to zero.   If $X$ has 
dimension one,  all rigid systems are cohomologically rigid  \cite[Thm.1.1.2, Cor.1.2.4]{Kat96}. As of today, we do not know of a single example of a rigid local system which is not cohomologically rigid, while experts believe they should exist.  The general references are  \cite{EG18} for a Betti to $\ell$-adic proof in the general quasi-projective case, and \cite{EG18a} for de Rham-crysalline  aspects of rigid local systems in the  projective case.

When $r=1$ and $X$ is projective, Simpson  considers  Zariski closed subloci $S\subset  \mathbf{{M}_B} (X,1)(\CC)$ (so $  \mathbf{{M}_B} (X,1)^{\rm irr} =  \mathbf{{M}_B} (X,1)$)
which have the property that viewed in $\mathbf{{M}_{dR}} (X,1)(\CC)$,  they are still Zariski closed. He calls them {\it bialgebraic subloci}. He proves \cite[ Thm.3.1 (c)]{Sim93}  that they necessarily are a finite union over the components of subtori  $T_i$ translated by points $s_i$. As $\mathbf{{M}_B} (X,r)$  is defined over $\ZZ$ and $\mathbf{{M}_{dR}} (X,r)$ over the field of definition of $X$, it makes sense to request that $X$, then $ S$ and its image in $\mathbf{{M}_{dR}} (X,1)$ are all defined over $\bar \QQ$. If this is the case, then the $s_i$ may be chosen to be torsion \cite[Thm.3.3]{Sim93}, in particular  the torsion points are Zariski dense on $S$. Note however that in the real topology, torsion points are located in the closed subspace defined by  the radius equal to one, thus are not dense.  Simpson's theorem has been generalized to the case when $X$ is quasi-projective in \cite{BW17}. 

In Section~\ref{sec:arithm} we report on our work with Moritz Kerz on arithmetic subloci of $\mathbf{{M}_B} (X,1)$ \cite{EK19}. We fix a prime number $\ell$. As integral $\ell$-adic points of $\mathbf{{M}_{B}} (X,1)$ are \'etale, that is 
\ga{}{ \mathbf{{M}_{B}} (X,1)(\bar \ZZ_
\ell) ={\rm Hom}_{\rm cts} ( \pi_1(X)^{\rm \acute{e}t}, \bar \QQ_\ell^\times),\notag}  where   $ \pi_1(X)^{\rm \acute{e}t}$ is the \'etale fundamental group, this set is acted on by the Galois group $G$ of the field of definition of $X$.  
In addition, it carries the  topology which is the restriction of the Zariski topology on 
$\mathbf{{M}_{B}} (X,1)(\bar \QQ_
\ell)$,  which we call Zariski topology. 
The theorem, in the spirit of Tate's conjecture,  then says that {\it arithmetic subloci}, that is Zariski closed  subloci  of $\mathbf{{M}_{B}} (X,1)(\bar \ZZ_
\ell)$ 
which are $G$-invariant, are a finite union over the components of subtori  $T_i$ translated by torsion points $s_i$. In addition, the $T_i$ are motivic.  This in particular applies to jumping loci.  The geometric condition on $X$ for the theorem to be true is very weak, one just assumes that the weights on $H_1(X)$ in the sense of Hodge theory  are strictly negative. For example normal varieties fulfil this condition, but also many non-normal  varieties as well. One easily constructs weight  zero examples with   non-dense torsion points. So the theorem is sharp.

We extract from the theorem a formulation of consequences in  Corollary~\ref{cor:EK} which enables  us to formulate in Section~\ref{sec:comm_ques} wished
 analogs for  higher rank arithmetic subloci.

\medskip

{\it Acknowledgements:} I thank Michael Groechenig and Moritz Kerz.
This survey is based on our joint work.  All ideas and questions formulated here  reflect our discussions. I thank Carlos Simpson
for several exchanges in the last years concerning the theorems presented in this report.  We acknowledge the deep influence of his work on ours.

\section{Simpson's integrality conjecture} \label{sec:int}

Let $X$ be a smooth connected  quasiprojective  complex variety, $j: X\hookrightarrow \bar X$ be a good compactification, $D=\bar X\setminus X$ be the  normal crossings divisor at infinity. We write $D=\cup_{i=1}^N D_i$ where the $D_i$ are the irreducible components. 
Fixing quasi-unipotent conjugacy classes $\K_i \subset GL_r(\CC)$ for each $i$,  and a rank one complex  local system $\sL$ of order $d$ on $X$, there is an algebraic  stack ${\mathbf M_B}(\K_i, \sL)$ of finite type defined over a number field $K$ which represents the functor from affine connected algebraic varieties $T$ to groupoids, which assigns  geometrically irreducible  representations of the topological fundamental group of $X^{\rm an}$ 
of rank $r$ with values in the linear automorphisms of a vector bundle on $T$, with monodromy at infinity falling in the conjugacy classes $\K_i$ and with determinant $\sL$ (\cite[Prop.2.1]{EG18}). In particular there are finitely many $0$-dimensional components.  They are defined over a finite extension of $K$. The associated complex points are call  {\it complex rigid local systems} with local monodromies $\K_i$ at infinity and prescribed isomorphism class of determinant $\sL$.  The Zariski tangent space at $\V \in {\mathbf M_B}(\K_i, \sL)$ is  the intersection cohomology 
$H^1(\bar X, j_{!*}  \E nd^0(\V))$ where $^0$ indicates the trace-free part (\cite[Prop.2.3]{EG18}).  Among the  complex rigid local systems there are those which are smooth on ${\mathbf M_B}(\K_i, \sL)$ which is equivalent to saying that 
 $H^1(\bar X, j_{!*}  \E nd^0(\V))=0$. They are called {\it cohomologically rigid local systems}.

 If $X$ has dimension one,  as already mentioned in the introduction,  all  rigid local systems are cohomologically rigid. {\it As of today we do not know of a single example in higher dimension where this is not the case}. Our theorem \cite[Thm.1.1]{EG18} yields a positive answer to Simpson's integrality conjecture for cohomologically rigid local systems.
 \begin{theorem} \label{thm:EGint}
  Let $X$ be a smooth connected   quasiprojective  complex variety.  Then cohomologically rigid complex  local systems with finite determinant and  quasi-unipotent  local monodromies around the components at infinity of a good compactification   are integral. 
 
 \end{theorem}
 \begin{proof}[Sketch of the proof.]
 The proof  is a {\it Betti to $\ell$-adic} proof. It uses  the existence of $\ell '$-adic companions of irreducible $\ell$-adic local systems on smooth varieties defined over finite fields, a deep theorem conjectured by Deligne  \cite[Conj.1.2.10]{Del80}, proven by L. Lafforgue in dimension one \cite[Chap.VII]{Laf02} as a consequence of the Langlands correspondence, and in higher dimension on a smooth variety by Drinfeld  \cite[Thm.1.1]{Dri12} by reduction to the dimension one case using geometry and representation theory.
 We do not know a pure characteristic zero proof of Theorem~\ref{thm:EGint}.
 
 A key ingredient of the proof is that there are finitely many local systems to consider (the rank is fixed), they are all defined over a number field, $K$ say,  and the possibly non-integral places are in finite number. So one can choose one place $\lambda$ of $K$ dividing the prime number $\ell$  which is integral for all of them. For  $p$ large, all those local systems are integral at the places dividing $p$. Grothendieck's specialization theorem on the tame fundamental group  enables us to descend
  those local systems  to $\bar \QQ_\ell$-adic local systems $\V_{{\bar \FF_p}}$ on $X_{\bar \FF_p}$ with the same determinant, 
  monodromy conditions at infinity and vanishing of $H^1(\bar  X_{\bar \FF_p}, j_{!*} \E nd^0 \V_{{\bar \FF_p}})$. 
  An argument essentially due to Simpson \cite[Prop.3.1]{EG18} shows that the $\bar \QQ_\ell$-adic local systems on $X_{\bar \FF_p}$ then descend to $X_{\FF_q}$ for some $q=p^s$.  Thus they have companions. Since $H^1(\bar  X_{\bar \FF_p}, j_{!*} \E nd^0 \V_{{\bar \FF_p}})$ is  precisely the pure weight one part of $\oplus_j H^j(X_{\bar \FF_p}, \E nd^0 \V_{{\bar \FF_p}}),$ it is preserved by the companion association. Since geometric irreducibility, the local mononodromy at infinity  and the finite determinants are preserved by the companion association, those companion $\ell'$-adic local systems, viewed as $\bar \ZZ_{\ell'}$-points of ${\mathbf M_B}(\K_i, \sL),$ are cohomogically rigid, thus the $\lambda$ to $\lambda'$ game induces a bijection of the finite set of cohomologically rigid local systems, thus they are integral at the place $\lambda'$.   This finishes the proof. 
 \end{proof}
 
 In fact, we understand more  facts on the {\it de Rham-crystalline} side towards Simpson's geometricity conjecture, if $X$ is projective.  
We have the following theorem \cite[Thm.1.6]{EG18a}.

 \begin{theorem}
 Let $X$ be a smooth projective variety over $\CC$, $\V$ be a rigid local system and $(E,\nabla)$ be its underlying connection. The connection  $(E,\nabla)$,  first  spread out over a ring of finite type over $\ZZ$,  then restricted to the formal scheme $\hat X_{W(\bar \FF_p)}$, where ${\rm Spec}(\bar \FF_p)\to {\rm Spec}(R)$ is a place of good reduction  for $(R, X, (E,\nabla)),$  has the structure of a crystal, which, tensor $\QQ$, is an  isocrystal with a Frobenius structure.
 
 \end{theorem}
 It is interesting to note that we do not need the cohomological condition in the  theorem.  By definition, the bundle $E$  on $\hat X_{W(\FF_p)}$ is respected by the connection. If $(E,\nabla)$ is motivic,  by \cite[3.1]{Kat72} the $p$-curvature of the mod $p$ reduction is nilpotent, which is to say that $(E,\nabla)$  on $\hat X_{W(\FF_p)}$ is a crystal (see \cite[Thm.2.20~2.6]{EG18a}). The Frobenius structure comes from the conjugate filtration. 
 \begin{proof}[Sketch of Proof]
 One way to see the nilpotency of the $p$-curvature is to count. We have finitely many $(E,\nabla)$ which come from rigid local systems of rank $r$. Then one introduces the other side of the Simpson correspondence, the moduli 
 $  \mathbf{{M}_{\rm Dol}} (X,r)^{\rm st}$ 
  of stable Higgs bundles. Since the correspondence yields a real analytic isomorphism between $  \mathbf{{M}_{\rm dR}} (X,r)^{\rm st}(\CC)$  and $\mathbf{{M}_{\rm B}} (X,r)^{\rm irr}(\CC)$, the number of isolated Higgs bundles is the same. On the other hand, going to characteristic $p$, which for $p$ large preserves this number, the Ogus-Vologodsky correspondence \cite{OV07} assigns  via the Cartier inverse operator $C^{-1}$ to a stable Higgs bundle with nilpotent curvature a connection. Thus it is enough to show that this connection is rigid as well, which may be extracted from 
  the splitting to order $(p-1)$ along the zero-section of the cotangent bundle of the Azumaya algebra  which is the center of the sheaf of differential operators \cite[Cor.2.9]{OV07}. 
 As for the Frobenius structure, one applies the theory of Lan-Sheng-Zuo of Higgs to de Rham flows \cite{LSZ13} produced by a rigid connection, which, given the finitely many rigid objects in a given rank, has to be cyclic, see \cite[Section~4]{EG18a}.
 \end{proof}

 \section{Arithmetic subloci in rank one} \label{sec:arithm}
 Our main theorem \cite[Thm.1.2]{EK19} yields an $\ell$-adic  analog of Simpson's theorem.
 \begin{theorem}  \label{thm:EKmain}
 Let $X$ be a complex quasi-projective variety defined over $\CC$. Let  $F\subset \CC$ be a field of finite type over which $X$ is defined, and $G$ be the Galois group ${\rm Aut}(\bar F/F)$ where $\bar F$ is the closure of $F$ in $\CC$. Fix a prime number $\ell$.  Let $S\subset  \mathbf{{M}_{B}} (X,1)(\bar \ZZ_\ell)$ be an arithmetic sublocus. 
  Then if the Hodge weights of $H_1(X, \ZZ)$ are strictly negative, $S$ is the finite union over its irreducible components of  subtori $T_i$ translated by torsion points $s_i$.  Moreover, the $T_i$ are motivic.  
 
 \end{theorem}
 Here a subtorus $T\subset \mathbf{{M}_{B}} (X,1)$ is called motivic if there is a  torsion free quotient Hodge structure $H_1(X, \ZZ) \twoheadrightarrow \pi'$ such that $T={\rm Hom}(\pi', \CC^\times)$, or equivalenty if there is a morphism $f: X\to Y$ of algebraic varieties such that $T=f^* \mathbf{{M}_{B}} (Y,1)^1$ where the subscript $^1$ indicates the connected component of $1$ (see \cite[Prop.7.3]{EK19}.)
 
 \begin{proof}[Sketch of proof] The proof uses the $\ell$-adic topology.  We reduce to the case where $S$ is irreducible. We fix a residual representation $\bar \xi$  
 of a point of $\xi\in S$.  If its Teichm\"uller  lift $[\bar \xi]$ in  ${\rm Hom}_{\rm cts}(\pi_1(X)^{\rm \acute{e}t}, \bar \QQ_\ell^\times) $ lies on $S$, we found a torsion point. If not, we replace $S$ by $[\bar \xi]^{-1}\cdot S$ so we may assume that $[\bar \xi]$ is the trivial representation. 
 Then  translating $S$ by a high power of $\ell$, we may assume that $S$ intersects a polydisk $${\rm Hom}_{\rm cts}(\pi_1(X)^{\rm \acute{e}t}, \bar \QQ_\ell^\times)(\rho) \subset {\rm Hom}_{\rm cts}(\pi_1(X)^{\rm \acute{e}t}, \bar \QQ_\ell^\times) $$ of radius $\rho$ where $\rho$ lies in the valued group of $\bar \QQ_\ell^\times$ on which
 the $\ell$-adic logarithm $\log$ is defined. It equates this analytic polydisc with the polydisc
 $${\rm Hom}_{\rm cts}(\pi_1(X)^{\rm \acute{e}t}, \bar \QQ_\ell)(\rho)=H^1(X, \bar \QQ_\ell)(\rho)\subset H^1(X, \bar \QQ_\ell).$$ 
 Let us assume for simplicity that $H_1(X, \ZZ)$ is pure, thus of strictly negative weight given our assumption. The goal is then to show that $\log(S\cap {\rm Hom}_{\rm cts}(\pi_1(X)^{\rm \acute{e}t}, \bar \QQ_\ell^\times)(\rho))$
 is homogeneous, thus in particular  contains $0$, which  implies  that  $ S\cap {\rm Hom}_{\rm cts}(\pi_1(X)^{\rm \acute{e}t}, \bar \QQ_\ell^\times)$ contains the trivial character, so the initial $S$ we started with contains a torsion point with given residual representation $[\bar \xi]$ (\cite[Prop.4.3]{EK19}).  In particular the torsion points of $S$ are Zariski dense, and we can choose one which in addition is smooth on $S$.
  After translation, $1\in S$ is a smooth point.  To prove the linearity, one considers the weights of $G$, see \cite[Prop.6.2]{EK19}. It uses a theorem of Bogomolov \cite{Bog80}, according to which  for $\alpha \in \ZZ_\ell^\times$ close to $1$ and not a root of unity, there is an element in $G$ which acts semi-simply on $H_1(X, \ZZ_\ell)$ with eigenvalues $\alpha^{-1}$. As $0$ is smooth, $ S\cap {\rm Hom}_{\rm cts}(\pi_1(X)^{\rm \acute{e}t}, \bar \QQ_\ell^\times)$ is linear. 
 Analytic linearlity implies algebraic linearity \cite[Lem.4.2]{EK19}. This concludes the proof if $H_1(X, \ZZ)$ is pure. 
  In the mixed case, one uses 
 Litt's generalization  \cite[Lem.2.10]{Lit18} of Bogomolov's theorem instead.
 
 Motivicity is proved using Faltings' theorem \cite[Thm.1]{Fal86}.
 
 \end{proof}
 \begin{example}
 Theorem~\ref{thm:EKmain} is sharp. Let $X$ be the union of two rational curves defined over $\QQ$ meeting in two points which are rational over $\QQ$. Then $H_1(X)=\ZZ$ has weight $0$, $\pi_1(X)=\hat \ZZ$ with trivial $G$-action. The character $\xi: \hat \ZZ \to \bar \ZZ_\ell^\times$ which assigns to $1$ a element in $\bar \ZZ_\ell^\times$ which is not a root of unity is thus not torsion. So $S= \mathbf{{M}_{B}} (X,1)(\bar \ZZ_\ell)$ does not fulfil the conclusion of the theorem.
 
 \end{example}

 We can draw from Theorem~\ref{thm:EKmain} obvious corollaries.   Let us remark that a point $\xi\in \mathbf{{M}_{B}} (X,1)(\bar \ZZ_\ell)$ is torsion if and only if its $G$-orbit is finite. 
 \begin{corollary} \label{cor:EK}
 In the situation of Theorem~\ref{thm:EKmain} 
 \begin{itemize}
 \item[1)] the  points in $S$  with finite $G$-orbit are Zariski dense;
 \item[2)] given $S_1, S_2  \subset \mathbf{{M}_{B}} (X,1)(\bar \ZZ_\ell)$  irreducible  arithmetic subloci,  a point $\xi\in S_1\cap S_2$ which is smooth on both $S_i$, such that $T_\xi S_1=T_\xi S_2\subset T_\xi \mathbf{{M}_{B}} (X,1)$, then
 $S_1=S_2$;
 \item[3)] given a prime $\ell'$, an algebraic isomorphism $\iota: \bar \QQ_\ell \xrightarrow{\cong} \bar \QQ_{\ell'}$, and an
 arithmetic sublocus $S\subset  \mathbf{{M}_{B}} (X,1)(\bar \ZZ_\ell)$, then $\iota(S)\subset  \mathbf{{M}_{B}} (X,1)(\bar \QQ_{\ell'})$  lies in 
 $ \mathbf{{M}_{B}} (X,1)(\bar \ZZ_{\ell'})$ and is arithmetic.  
  \end{itemize}
 
 \end{corollary}
 \begin{remark} \label{rmk:EK}
 \begin{itemize}

 \item[1)] 1) is a direct consequence of the structure of commutative algebraic group schemes over an algebraically closed characteristic zero field which are extensions of a finite group scheme by a torus: its torsion points are dense.  We remark that 
 the converse to 1) is not true.  For example, if $X$ is an elliptic curve, then  $\mathbf{{M}_{B}} (X,1)$ is a two dimensional torus, but no dimension one subtorus is motivic. However any torus is the Zariski closure of its  torsion points. 
  \item[2)] 2) comes from the fact that a subtorus is determined by its Zariski tangent space at $1$.  The condition that $\xi$ is smooth on either $S_i$ is superfluous as it follows from the theorem, but we keep this formulation in view of Section~\ref{sec:comm_ques}. 
 \item[3)] 3) comes from the motivicity part of the theorem. 
 \end{itemize}
 \end{remark}
 
\section{Comments and questions} \label{sec:comm_ques}
\subsection{Integrality}
To bypass the cohomological argument in Section~\ref{sec:int}, and therefore to prove the integrality conjecture completely (should there really exist rigid local systems which are not cohomologically rigid), one has  to understand how the formal neighbourhood of $0$ in  $$\{x\in H^1(\bar X_{\bar \FF_p}, j_{!*}\E nd^0(\V)), \ 0=x\cup x\in H^2(\bar X_{\bar \FF_p},  j_{!*}\E nd^0(\V))\}$$  varies in the companion correspondence. Indeed, interpreted in Betti cohomology over $\CC$, this is the equation of the formal ring of $ \mathbf{{M}_{B}} (X,r)^{\rm irr}$ at $\V$ \cite[Cor.2.4]{Sim92}.  We do not know how to do this. 
\subsection{Higher rank arithmetic subloci}
In higher rank, given a prime number $\ell$, we define similarly {\it arithmetic subloci} of $\mathbf{{M}_{B}} (X,r)^{\rm irr}$ as Zariski closed subloci of 
$\mathbf{{M}_{B}} (X,r)^{\rm irr}(\bar \ZZ_\ell)$ which are $G$-invariant.  Here again the Zariski topology on 
$\mathbf{{M}_{B}} (X,r)^{\rm irr}(\bar \ZZ_\ell)$ is by definition the restriction of the Zariski topology on $\mathbf{{M}_{B}} (X,r)^{\rm irr}(\bar \QQ_\ell)$.
Then we ask at least for $X$ smooth projective over $\CC$ whether Corollary~\ref{cor:EK} is true,  where we replace $\mathbf{M_B}(X,1)$ by $\mathbf{M_B}(X,r)^{\rm irr}$, or  by   $\mathbf{M_B}(X,r, \sL)^{\rm irr}$ fixing the determinant $\sL$.   In the non-proper case, we ask the same question replacing $\mathbf{{M}_{B}} (X,r)^{\rm irr}$ which is only locally   of finite type by the moduli $\mathbf{{M}_{B}} (\K_i, \sL)$ as in Section~\ref{sec:int} or  by the moduli 
 $\mathbf{{M}_{B}} (\K_i)$ which is defined similarly without fixing the determinant. 
As we know very little on the topological fundamental group of a smooth quasi-projective complex variety, we are very far from understanding 1). For example, even if $S$ is the whole moduli space,  we do not even know the existence of a single point with finite Galois orbit. 
2) is accessible, but useless as we do not have 1). As for 3), a positive answer would prove Theorem~\ref{thm:EGint} without cohomological assumption.


\begin{thebibliography}{BBDE04}

 \bibitem[Bog80]{Bog80} Bogomolov, F.: {\it Sur l'alg\'ebricit\'e des repr\'esentations $\ell$-adiques}, C.R. Acad. Sc. Paris {\bf 290} (1980), 701--703.
 \bibitem[BW17]{BW17}  Budur, N., Wang, B.: {\it Absolute sets and the Decomposition Theorem},
   {\url{ arXiv:1702.06267}}, to appear in Ann. \'Ec. Norm. Sup.
\bibitem[Del70]{Del70} Deligne, P.: {\it \'Equations diff\'erentielles \`a points singuliers r\'eguliers}, Lecture Notes in Mathematics {\bf 163} (1970), Springer Verlag.
 \bibitem[Del80]{Del80} Deligne, P.: {\it La conjecture de Weil II}, Publ. math. de l'I.H.\'E.S {\bf 52} (1980), 137-252.
\bibitem[Dri12]{Dri12} Drinfeld, V.: {\it On a conjecture of Deligne}, 
Mosc. Math. J. {\bf 12} (2012) no. 3, 515--542. 
 \bibitem[EG18]{EG18} Esnault, H., Groechenig, M.: {\it  Cohomologically rigid local systems and integrality},  Selecta Mathematica {\bf 24} (2018) 5, 4279--4292.
 \bibitem[EG18a]{EG18a}   Esnault, H., Groechenig, M.:   {\it Rigid connections and $F$-isocrystals},  preprint 2018, 44 pages,  {\url{http://page.mi.fu-berlin.de/esnault/preprints/helene/126_esn_gro.pdf}}.
 \bibitem[EK19]{EK19} Esnault, H., Kerz, M.: {\it Arithmetic subspaces of moduli spaces of rank one local systems}, preprint 2019, 21 pages,  {\url{http://page.mi.fu-berlin.de/esnault/preprints/helene/133_esn_kerz.pdf}}
 \bibitem[Fal86]{Fal86} Faltings, G.: {\it Complements to Mordell},  in Faltings, G., W\"ustholz, G.: {\it  Rational Points}, Aspects of Math. {\bf E6}, Vieweg Verlag  (1986) 203--227.
 \bibitem[Kat72]{Kat72} Katz, N.: {\it Algebraic Solutions of Differential Equations (p-curvature and the Hodge Filtration)},
Invent. math. {\bf 18}  (1972), 1--118.
 \bibitem[Kat96]{Kat96} Katz, N.: {\it Rigid local systems}, Princeton University Press 1996. 
 \bibitem[Laf02]{Laf02}  Lafforgue, L.: {\it Chtoucas de Drinfeld et correspondance de Langlands}, Invent. math. {\bf 147} (2002), no 1, 1--241.
\bibitem[LSZ13]{LSZ13} Lan, G., Sheng, M., Zuo, K.:  {\it Semistable Higgs bundles, periodic Higgs bundles and representations of algebraic fundamental groups}, to appear in J. Eur. Math. Soc., {\url{ https://arxiv.org/pdf/1311. 6424.pdf}} 
 \bibitem[Lit18]{Lit18} Litt, D.: {\it Arithmetic representations of fundamental groups I}, Inventiones math. {\bf 214} (2018), 605--639.
 \bibitem[OV07]{OV07} Ogus, A., Vologodsky, V.: {\it Nonabelian Hodge theory in characteristic $p$}, Publ. math. de l'I.H.\'E.S.  {\bf 106} (2007), 1--138.
\bibitem[Sim92]{Sim92} Simpson, C.:  {\it Higgs bundles and local systems}, Publ. math. de  l'I.H.\'E.S {\bf 75} (1992), 5--95.
 \bibitem[Sim93]{Sim93} Simpson, C.: {\it Subspaces of moduli spaces of rank one local systems}, Annales de l'\'E. N. S. 4\`eme s\'erie {\bf 26} 3 (1993), 361--401.
 \bibitem[Sim94]{Sim94} Simpson, C.:  {\it  Moduli of representations of the fundamental group of a smooth projective variety. I.}, Publ. math. de  l'I.H.\'E.S {\bf 79} (1994), 47--129.

\end{thebibliography}
 \end{document}